\documentclass[11pt,reqno]{amsart}

\usepackage{amsmath,amsthm,amssymb,enumerate}

\theoremstyle{plain}

\newtheorem{theorem}{Theorem}[section]

\newtheorem{corollary}[theorem]{Corollary}
\newtheorem{proposition}[theorem]{Proposition}
\newtheorem{Lemma}[theorem]{Lemma}

\theoremstyle{definition}

\numberwithin{equation}{section}
\begin{document}

\title{Spectral analysis of Neumann-Poincar\'e operator}

\author[Ando Kang Miyanishi Putinar]{Kazunori Ando, Hyeonbae Kang, Yoshihisa Miyanishi and Mihai Putinar}

\address{Department of Electrical and Electronic Engineering and Computer Science, Ehime University, Ehime 790-8577, Japan, {\tt ando@cs.ehime-u.ac.jp}}

\address{Department of Mathematics, Inha University, Incheon
402-751, S. Korea, {\tt hbkang@inha.ac.kr}}

\address{University of California at Santa Barbara, CA 93106, USA and
Newcastle University, Newcastle upon Tyne NE1 7RU, UK, {\tt mputinar@math.ucsb.edu, mihai.putinar@ncl.ac.uk}}

\address{Center for Mathematical Modeling and Data Science, Osaka University, Japan,
{\tt miyanishi@sigmath.es.osaka-u.ac.jp}}

 \maketitle

\begin{abstract}
This is a survey of accumulated spectral analysis observations spanning more than a century, referring to the double layer potential
integral equation, also known as Neumann-Poincar\'e operator. The very notion of spectral analysis has evolved along this path. Indeed, the quest for solving this specific singular integral equation, originally aimed at elucidating classical potential theory problems, has inspired and shaped the development of theoretical spectral analysis of linear transforms in XX-th century. We briefly touch some marking discoveries into the subject, with ample bibliographical references to both old, sometimes forgotten, texts and new contributions. It is remarkable that applications of the spectral analysis of the Neumann-Poincar\'e operator are still uncovered nowadays, with spectacular impacts on applied science. A few modern ramifications along these lines are depicted in our survey.

\end{abstract}
 \tableofcontents

\section{Introduction} Layer potentials entered early into the main stream of mathematics as field theoretical entities. Their role in forging the main
problems of Mathematical Physics of XIX-th Century is well known. The dawn of modern spectral analysis is also rooted in the attempts to exploit layer potentials as
integral transforms with a singular kernel. Less known for the general reader is the notable impact of layer potentials on current investigations in applied fields of high technological impact, such as photonics and phononics, composite materials, medical imaging.

The present  survey offers a few snapshots on the spectral analysis of the double layer potential, also known as the Neumann-Poincar\'e (in short NP) operator. We record, with ample bibliographical references, Poincar\'e's 1897 breakthrough
vision, together with some of its aftershocks, to jump to some spectacular advances made during the last decade. More specifically, we explain how the balance of inner-outer harmonic field energies led Poincar\'e to an early discovery of a discrete spectrum of real critical values. We place this pictures into the modern theory of spectral analysis of compact, symmetrizable Hilbert space operators and furthermore, we link it to quasi-conformal theory and specific integral transforms arising in the function theory of a complex variable. Then
we comment on Carleman's doctoral dissertation which contains the first explicit spectral resolution of a non-compact NP operator, associated to a planar domain with one corner. We mention a couple of recent contributions clarifying at the level of today's mathematical rigor Carleman's foresight.  Finally, we explain in some detail recent discoveries referring to the asymptotics of eigenvalues distribution of the NP operator on 2D or 3D domains. A separate section illustrates recent spectral computations arising in the double layer potential associated to an elasticity problem. The contents above gives a glimpse on the topics touched in our article.

Our survey is far from being exhaustive. We have selected only a few relevant aspects of a highly multifaceted subject, with the aim at connecting some undergoing research to classical contributions. Any omission is non-intentional. Parts of this survey have appeared in one or another of our publications. We reproduce relevant paragraphs of our prior works without quotations.

\bigskip

{\bf Acknowledgement.} The authors thank for hospitality the Simion Stoilow Mathematical Institute of the Romanian Academy for hosting them in the Summer of 2019 and offering them the opportunity to progress on their joint research and craft the present survey.

\section{Prerequisites of potential theory}

We recall in this section some terminology and basic facts of Newtonian potential theory. Details can be found in \cite{AK-book-07,KPS-ARMA-07}.
We warn the reader that there is no consensus in the vast literature on the subject of signs and constants in the definitions of potentials. We hope this will not be a cause of confusion.

\subsection{Potentials and jump formulae}

Let $d \geq 2$ be an integer and $\Omega$ be a bounded domain in $\mathbf R^d$ with boundary $\Gamma$.
For the time being we assume that $\Gamma$ is at least $C^2$-smooth, or as Radon emphasizes in his geometric approach \cite{Radon2}, the curvature of $\Gamma$ is continuous. The $(d-1)$-dimensional surface measure on $\Gamma$
is $d\sigma$ and the unit outer normal to a point $y \in \Gamma$
will be denoted $n_y$.

We let $E(x,y)=E(x-y)$ stand for the normalized
Newtonian kernel:
\begin{equation}
E(x,y) = \left\{ \begin{array}{cc}
           \frac{1}{2\pi} \log \frac{1}{|x-y|}, & d=2,\\
           c_d |x-y|^{2-d}, & d \geq 3,
         \end{array}\right.
\end{equation}
where $c_d^{-1}$ is the surface area of the unit sphere in
$\mathbf R^d$. The signs were chosen so that $\Delta E = -\delta$
(Dirac's delta-function).

We associate to a $C^2$-smooth function (density, in physical terms) $f(x)$ on $\Gamma$ the fundamental potentials:
the {\it single and double layer potentials} in $\mathbf R^d$; denoted $S_f$ and $D_f$ respectively:
\begin{equation}
S_f(x) = \int_\Gamma E(x,y) f(y) d\sigma(y),\ \ \
D_f(x) = \int_\Gamma \frac{\partial}{\partial n_y} E(x,y) f(y) d \sigma(y).
\end{equation}
The surface $\Gamma$ divides $\mathbf R^d$ into two domains $\Omega = \Omega_i$
(interior to $\Gamma$) and the exterior $\Omega_e$. Thus the potentials above
define pairs of functions $(S_f^i, S_f^e)$ and $(D_f^i,D_f^e)$ which are harmonic
in $\Omega_i$ and $\Omega_e$ respectively.

As is well known from classical potential theory (cf. \cite{Kellog}, \cite{Sobolev})
denoting by $S_f^i(x), \frac{\partial}{\partial n_x} S_f^i (x)$ (and corresponding symbols with superscript e) the
limits at $x \in \Gamma$
from the interior (or exterior), the following relations (known as the jump formulas for the
potentials) hold for all $x \in \Gamma$:
\begin{equation}\begin{array}{l}
S_f^i(x) = S_f^e(x);\\
\\
\frac{\partial}{\partial n_x} S_f^i(x) = \frac{1}{2} f(x) + \int_\Gamma  \frac{\partial}{\partial n_x}
E(x,y) f(y) d\sigma(y);\\
\\
 \frac{\partial}{\partial n_x} D_f^i(x) =  \frac{\partial}{\partial n_x} D_f^e(x);\\
 \\
 D_f^i(x) = -\frac{1}{2} f(x) + \int_\Gamma  \frac{\partial}{\partial n_y} E(x,y) f(y) d\sigma(y);\\
 \\
 \frac{\partial}{\partial n_x} S_f^e(x) =  -\frac{1}{2} f(x) + \int_\Gamma  \frac{\partial}{\partial n_x}
E(x,y) f(y) d \sigma(y);\\
\\
D_f^e(x) = \frac{1}{2} f(x) + \int_\Gamma  \frac{\partial}{\partial n_y} E(x,y) f(y) d\sigma(y).
\end{array}
\end{equation}

Rather direct computations (see for instance
Chapter II in \cite{Kellog}, Chapters 18-19 in \cite{Sobolev} or \cite{Mazya}) show that
the integral kernels
$$ K(x,y):=  -\frac{\partial}{\partial n_y} E(x-y); \ \  K^\ast(x,y) =  -\frac{\partial}{\partial n_x} E(x-y)$$
satisfy on $\Gamma$ the following estimates, for $d \geq 3$:
\begin{equation}\begin{array}{l}
|K(x,y)| = O(\frac{1}{|x-y|^{d-2}}), \ \ x,y \in \Gamma, x \neq y,\\
|K^\ast(x,y)| = O(\frac{1}{|x-y|^{d-2}}), \ \ x,y \in \Gamma, x
\neq y.
\end{array}
\end{equation}

For $d=2$, due to the fact that $\log|z-w|$ is the real part of
a complex analytic function $\log(z-w) = \log|z-w| + i \arg(z-w), \ z,w \in \Gamma$, and by Cauchy-Riemann's
equations one obtains
$$ K(z,w) = \frac{\partial}{\partial \tau_w} \arg(z-w),$$
where $\tau_w$ is the unit tangent vector to the curve $\Gamma$. Thus, on any
smooth curve $\Gamma \subset \mathbf R^2$, the kernels $K(z,w)$ and $K^\ast(z,w)$ are  uniformly bounded.

Returning to the general $d$-dimensional case, we define on $L^2(\Gamma) = L^2(\Gamma, d\sigma)$ the
{\it Neumann-Poincar\'e  operator} $K$:
\begin{equation}
(Kf)(x) = 2 \int_\Gamma K(x,y) f(y) d\sigma(y), \ \ f \in L^2(\Gamma, d\sigma).
\end{equation}
The $L^2$ adjoint $K^\ast$ will be an integral operator with
kernel $K^\ast(x,y)$.  The nature of the diagonal singularity of
the kernel $K(x,y)$ shows  that $K$ is a compact operator in the
Schatten-von Neumann class $\mathcal{C}^p(L^2(\Gamma)), p > d-1,$
see \cite{Kellog}. Since the kernel $K$ is bounded when $d=2$, it
is Hilbert-Schmidt on any smooth planar curve. We will show in the
next section that $K^\ast$ is symmetrizable, that is $K^\ast$
becomes self-adjoint with respect to a different (incomplete)
inner product on $L^2 (\Gamma)$.

Similarly, the linear operator $$ Sf = S_f|_\Gamma, \ \ f \in
L^2(\Gamma),$$ turns out to be bounded (from $L^2(\Gamma)$ to the
same space). Remark that the representing kernel $E(x,y)$ of $S$
is pointwise  non-negative for $d\geq 3$. With these conventions
the jump formulas become, as functions on $\Gamma$:
\begin{equation}\begin{array}{l}
  {S_f}^i = S_f^e = Sf;\\
  \\
 \partial_n S_f^i = \frac{1}{2} f -\frac{1}{2} K^\ast f;\\
 \\
\partial_n S_f^e = -\frac{1}{2} f -\frac{1}{2} K^\ast f;\\
\\
D_f^i = -\frac{1}{2} f -\frac{1}{2}Kf;\\
\\
D_f^e = \frac{1}{2} f -\frac{1}{2}Kf.\\
\end{array}
\end{equation}
Above, and always in this paper $n$ designates the {\it outer} normal
to $\Omega$.
\bigskip

Thus, to solve Dirichlet's problem $\Delta u =0$ in $\Omega_i$, with continuous boundary data
$u|_\Gamma = \phi$ it is sufficient to solve the integral equation
$$ f + K f = -2\phi. $$
Then $u = D_f$ would be the harmonic solution. Carl Neumann \cite{Neumann} remarked that
the infinite sum, known today as Neumann series,
$$ -2\phi + 2K \phi - 2 K^2 \phi + \ldots, $$
if convergent, provides the density $f$. Great deal of efforts were concentrated during the last two decades of XIX-th Century
at enlarging the classes
of shapes $\Gamma$ for which Neumann series converges.

Since the operator $K$ is compact, the invertibility of $I+K$ led Fredholm \cite{Fredholm} to his infinite determinant
$\det (I + \lambda K)$ study which is one of the uncontested foundations of spectral analysis of integral equations. Soon afterwards
Hilbert \cite{Hilbert}, F. Riesz \cite{Riesz,RN} and many of their contemporaries laid the theoretical framework streaming from these computations.
In the introduction of his influential monograph Plemelj \cite{Plemelj} states in 1910 that ``{\it one of the most important tasks of modern Potential Theory
is to explore the invertibility of boundary singular integral equations arising from single and double layer potentials}".

The reader interested in tracking such fascinating historical aspects of Potential Theory can consult the German Encyclopaedia articles \cite{Lichtenstein,BM,HT}
and the circulating monographs of the period such as Korn's \cite{Korn2}.

Nowadays we have well developed and fully understood definitions, jump formulae, boundedness and the like, of single and double layer potentials on Lipschitz boundaries.
The technicalities related to such generalizations on "rough" boundaries are by no means trivial, see \cite{CMM82,DKV-Duke-88,Mazya}.

\subsection{Poincar\'e's variational principle}

C. Neumann's series method was successfully applied towards the solvability of Dirichlet's problem on convex domains.
A more general method to solve Dirichlet's problem was discovered by Poincar\'e and called by him ``balayage''. However,
the convergence of Neumann series and in general the invertibility of the double layer potential operator on the boundary
were computationally more accessible techniques. In a long and technical memoir of 1897, Poincar\'e attacked the convergence of the Neumann series
on non-convex boundaries \cite{Poi1}. At the end of this memoir he remarks yet another approach to solve Dirichlet's problem, by considering the critical values of a
balance of energies (inner, versus outer) of the field produced by a boundary charge. Poincar\'e concludes his remarks by the following words:
{\it "After having established these results
[concerning convergence of the Neumann series] rigorously, I felt
obliged in the two final chapters to give an idea of the insights
which initially led me to foresee these results. I thought that,
despite their lack of rigor, these could be useful as tools for
research insofar as I had already used them successfully once."}

The new insight was immediately noticed. For instance in 1900
Steklov praises the central role of Poincar\'e's new extremal functions \cite{Stekloff1,Stekloff2} stating
{\it ``the solvability of all fundamental problems of Mathematical Physics is reduced to the proof of the
fundamental theorem [of Poincar\'e]''} \cite{Stekloff1} pp. 224. And he gives ample evidence to this bold assertion,
analyzing indeed current fundamental boundary value problems based on strong convergence of orthogonal series decompositions or, in today's terminology, min-max characterization of characteristic values of quadratic forms depending on infinitely many variables.
See also \cite{Stekloff3} for a closely related topics, now bearing the name of Steklov's problem \cite{GP}.

Although Poincar\'e's program and in particular his ``fundamental functions" entered into the main circuit of Potential Theory, cf.
\cite{Lichtenstein}, only sporadic and indirect reverberations of it were recorded for the last century. Poincar\'e's variational principle was formulated in modern terms in a recent article \cite{KPS-ARMA-07} which serves as the root and guide of the present section.

The setting is as follows, still fixing a smooth bounded domain $\Omega$ in Euclidean space.
Let $ H$ be the space of harmonic functions on $\Omega_i \cup \Omega_e$, vanishing at
infinity and with finite energy seminorm:
\begin{equation}
 \| h \|_{ H}^2 =  \int_{\Omega_i \cup \Omega_e} |\nabla h|^2 dx.
\end{equation}
Note that only  locally constant functions are annihilated by this seminorm.
It will be necessary to distinguish between the two restrictions
of $h$ to the inner and outer domain; we denote $h= (h_i, h_e)$
where $h_i = h|_{\Omega_i}$ and similarly $h_e = h|_{\Omega_e}$.
In virtue of Poincar\'e's inequality the functions $h_i$ and $h_e$
are in the Sobolev $W^{1,2}$-spaces of the corresponding domains.
To simplify notation we put henceforth $W^s = W^{s,2}$.

We can regard an element $h \in   {H} \subset \mathcal{\mathcal D}'(\mathbf R^d)$ as a distribution defined on the whole
space. Then $\Delta h = \mu \in \mathcal{\mathcal D}'_\Gamma(\mathbf R^d)$ (the lower index means ${\rm supp}( \mu) \subset \Gamma$)
and, by a slight abuse of notation
\begin{equation}
-h(x) = S_\mu(x) = \int_\Gamma E(x,y) d\mu(y), \ \ x \in \mathbf{R}^d \setminus \Gamma.
\end{equation}
If the distribution $\mu$ is given by a smooth function times the
surface measure of $\Gamma$, then $h = S_\mu$ and, as a single layer potential,
$h_i|_\Gamma = h_e|_\Gamma$. Our next aim is to identify the
closed subspace of ${H}$ characterized by the latter
matching  property.

By assumption the surface $\Gamma$ is smooth. Hence there are linear
continuous (restriction to the boundary) trace operators
$$ {\rm Tr} : W^{1}(\Omega_{i,e}) \longrightarrow W^{1/2}(\Gamma).$$
Moreover, the trace operator from each side of $\Gamma$ is
surjective (and hence it has a continuous right inverse), see for
instance \cite{LM}. We will denote in short
$$ h|_\Gamma = {\rm Tr} \ h.$$
If $d \geq 3$, then for any function $f \in W^{1/2}(\Gamma)$ there
exist solutions $(h_i, h_e) \in {H}$ to the inner and
outer Dirichlet problems with boundary data $f$: $h_i|_\Gamma =
h_e|_\Gamma = f$, see \cite{Landkof}.

In the case
$d=2$ the additional assumption $\int_\Gamma f d\sigma = 0$ must be made, to assure the existence
of $h_e$ with $h_e(\infty) =0$ and finite energy, see \cite{Landkof} .

A series of remarkable identifications, simple consequences of Stokes theorem, are in order.
We refer for proofs to Landkof's monograph \cite{Landkof} or the cited article \cite{KPS-ARMA-07}.

For every $f \in L^2(\Gamma)$,
\begin{equation}
\langle Sf,f \rangle = \| S_f \|^2_{ H}.
\end{equation}

 Hence $S$ is a non-negative, linear and bounded self-adjoint operator on $L^2(\Gamma)$.
From this observation, the identification of single layer potentials in the energy space $H$ follows:

\begin{proposition} Assume $d \geq 3$ and let $h = (h_i,h_e) \in {H}$. Then
$ h_i|_\Gamma  =  h_e|_\Gamma$ if and only if there exists $\rho \in W^{-1/2}(\Gamma)$
such that $h = S_\rho$.
\end{proposition}

The case $d=2$ requires again the additional assumption that $\rho(\mathbf 1)=0$.
Otherwise $S_\rho$ would not have a square summable gradient on the exterior domain.

Henceforth we define the {\it subspace of single layer potentials} by
$$ {\mathcal S} = \{ h \in {H}, \ h_i|_\Gamma = h_e|_\Gamma \}.$$
The orthogonal complement in ${H}$ will be denoted
${\mathcal D} = {\mathcal S}^\perp$ and we will identify this with the {\it space of
double layer potentials} belonging to ${H}$. Indeed, some standard Sobolev space considerations lead to
the following statement.

\begin{proposition} Let $h = (h_i,h_e) \in {H}$. The following conditions are equivalent:

a) $h \in {\mathcal D} (= {\mathcal S}^\perp);$

b) $\frac{\partial h_i}{\partial n} = \frac{\partial h_e}{\partial
n}$ (in $W^{-1/2}(\Gamma));$

c) There exists $f \in W^{1/2}(\Gamma)$ such that $ h = D_f$.

In this case $f = h_e-h_i$.
\end{proposition}

The only elements of ${H}$ annihilated by the energy
seminorm are scalar multiples of $(\mathbf{1}, 0)$. This is the
double layer potential of the constant function, and is,
therefore, orthogonal to ${\mathcal S}$. By the boundary formula $
(D_f)_e = \frac{1}{2} f - \frac{1}{2} K$ we infer $K{\mathbf 1} =
\mathbf{1}.$

Another distinguished element of ${H}$ is provided by the
so called equilibrium distribution $\rho$ on $\Gamma$; namely $(1,h) \in
{\mathcal S}$, that is $S \rho = 1$ and $h = S^e_\rho$.

A central role in Poincar\'e's program is played by the following result, known as {\it Plemelj's symmetrization principle}, see \cite{Plemelj, Korn1, Korn3} .

\begin{Lemma} The operators $S,K : L^2(\Gamma) \longrightarrow L^2(\Gamma)$ satisfy the identity
\begin{equation}
KS = SK^\ast .
\end{equation}
\end{Lemma}

In other terms
$$ \langle S K^\ast f, g \rangle_{2,\Gamma} = \langle S f, K^\ast g \rangle_{2,\Gamma}, \ \ f,g \in L^2(\Gamma).$$
Since we can identify  isometrically $\sqrt{\mathcal S} L^2(\Gamma) = W^{1/2}(\Gamma)$ and consequently $S W^{-1/2} (\Gamma) = W^{1/2}(\Gamma)$,
Plemelj's symmetrization formula states that the adjoint of Neumann-Poincar\'e (NP) operator $K^\ast$ is bounded and self-adjoint in $W^{-1/2}.$
Hence it has real spectrum! Much more can be said about its spectrum following Poincar\'es steps.

 The prehilbertian space ${H}$ possesses two natural direct sum decompositions:
$$ {H} = {\mathcal S} \oplus {\mathcal D} = {H}_i \oplus {H}_e.$$
By definition, the latter subspaces are
$$ {H}_i = \{ (h_i,0) \in {H}\}, \ \ \ {H}_e = \{ (0,h_e) \in {H}\}.$$
Let $P_s,P_d,P_i,P_e$ be the corresponding orthogonal projections.
The only subspace ${N} = \mathbf{C} (\mathbf{1}, 0)$
annihilated by the seminorm satisfies:
$$ {N} \subset {\mathcal D} \cap {H}_i.$$
Sometimes we will prefer to work within a Hilbert space, and then we will replace tacitly
${H}$ by ${H} \ominus {N}$.

Recall that the boundary single layer potential $S$ is an $L^2$-positive operator mapping
$W^{-1/2}(\Gamma)$ onto $W^{1/2}(\Gamma)$. The $L^2$ pairing between the two Sobolev spaces is
still denoted $\langle S\rho, f\rangle_{2,\Gamma}, \ \rho \in W^{-1/2}(\Gamma), f \in W^{1/2}(\Gamma)$.

For a boundary charge $f \in W^{-1/2}(\Gamma)$ define the inner, respectively, outer energy functionals:
$$ J_i[f] = \int_{\Omega_i} |\nabla S_f|^2 dx,$$
$$ J_e[f] = \int_{\Omega_e} |\nabla S_f|^2 dx. $$

Poincar\'e proposes to analyze the characteristic values of the (Rayleigh) quotient of quadratic forms
$$ \frac {J_e[f]-J_i[f]}{J_e[f]+J_i[f]}$$
and predicts that they fill a discrete spectrum. Hence, we guess in modern terms that some relative compactness
is responsible for this phenomenon. Indeed,
 \begin{equation}
 \frac {J_e[f]-J_i[f]}{J_e[f]+J_i[f]} = \frac{\langle (P_e-P_i)S_g, S_g \rangle_{H}}{\| S_g \|^2_{H}} =
\frac{\langle KS g, g \rangle_{2,\Gamma}}{\langle Sg,g \rangle^2_{2,\Gamma}}.
 \end{equation}

In other terms, Poincar\'e's quotient of quadratic forms is precisely the ``angle operator'' between the two orthogonal decompositions of the energy space.
Plemelj's symmetrization principle and a non-trivial observation of general nature, to be recalled in a subsequent section, lead to the following
theorem, whose main points were foreseen by
Poincar\'e \cite{Poi1}.

\begin{theorem} Let $\Omega \subset \mathbf R^d$ be a bounded domain with smooth boundary $\Gamma$
and let $\Omega_e = \mathbf R^d \setminus \overline{\Omega}.$
Let $S_\rho$ denote the single layer potential of a distribution $\rho \in W^{-1/2}(\Gamma),
\ \ (\rho(\mathbf 1)=0$ in case $d=2$).

Define successively, as long as the maximum is positive, the energy quotients
\begin{equation}
\lambda_k^+ = \max_{ \rho \perp \{ \rho_0^+, ..., \rho_{k-1}^+\}}
\frac{\| \nabla S_\rho \|^2_{2,\Omega_e} - \| \nabla S_\rho \|^2_{2,\Omega}} {\| \nabla S_\rho \|^2_{2,\mathbf R^d}},
\end{equation}
where the orthogonality is understood with respect to the total
energy norm. The maximum is attained at a smooth distribution
$\rho^+_k \in W^{1/2}(\Gamma).$

Similarly, define
\begin{equation}
\lambda_{k}^- =
\min_{ \rho \perp \{ \rho_1^-, ..., \rho_{k-1}^-\}}
\frac{\| \nabla S_\rho \|^2_{2,\Omega_e} - \| \nabla S_\rho \|^2_{2,\Omega}} {\| \nabla S_\rho \|^2_{2,\mathbf R^d}}.
\end{equation}
The minimum is attained at a smooth distribution $\rho_k^- \in W^{1/2}(\Gamma)$.

The potentials $S_{\rho_k^\pm}$ together with all $S_\chi \in \ker
K  (\chi \in W^{-1/2}(\Gamma)), $  are mutually orthogonal and
complete in the space of all single layer potentials of finite
energy.
\end{theorem}

We refer to \cite{KPS-ARMA-07} for the proof and more details.

The stronger than expected regularity of the eigenfunctions $(
\rho_k^\pm \in W^{1/2}(\Gamma))$ will be explained in Section \ref{2-norms}  as a consequence of a general theory of symmetrizable operators
in a space with two norms.

 The equilibrium distribution of $\Omega$
provides the first function $\rho_0^+$ in this process: $ S
\rho_0^+ = 1, S_{\rho_0^+}|_\Omega = 1.$ The first eigenvalue is
always $\lambda_0^+ =1$ and has multiplicity equal to one.

In conclusion, the spectrum of the Neumann-Poincar\'e operator $K$, multiplicities included, coincides with
the spectrum $(\lambda_k^\pm)$ of Poincar\'e's variational problem, together with possibly the point zero. The
extremal distributions for the Poincar\'e problem are exactly the eigenfunctions of $K$. And, as we shall see, similarly for $K^\ast$.

In practice it is hard to work directly with the NP operator on $L^2(\Gamma)$. Instead, the following
interpretation of the extremal solutions to Poincar\'e's variational problem is simpler and more flexible.
This also goes back to Poincar\'e's memoir \cite{Poi1}, and it was constantly present in the works of potential theory
in the first decades of the twentieth century, cf. for instance \cite{Plemelj}.

Let us start with an eigenfunction $f \in L^2(\Gamma)$ of the operator $K^\ast$. Then
$$K^\ast f = \lambda f \ \Rightarrow K Sf = S K^\ast f = \lambda S f,$$
and by the jump formulas (6)
$$ \partial_n S_f^i = \frac{1-\lambda}{2} f, \ \ \ \partial_n S_f^e = \frac{-1-\lambda}{2} f.$$
The associated energies are
$$ J_i[f] = \int_{\Omega_i} |\nabla S_f|^2 dx = \frac{1-\lambda}{2} \langle Sf, f \rangle,$$
$$ J_e[f] = \int_{\Omega_e} |\nabla S_f|^2 dx = \frac{1+\lambda}{2} \langle Sf, f \rangle.$$
To verify our computations, simply note that
$$ \frac{J_e[f]-J_i[f]}{J_e[f]+J_i[f]} = \lambda.$$

The characteristic feature of the above single layer potential $S_f$ is encoded in
the following statement.

\begin{proposition}
 A pair of harmonic functions $(h_i,h_e) \in {H}$ represents an extremal potential
for Poincar\'e's variational problem (distinct from the equilibrium distribution)
if and only if, there are non-zero constants $\alpha, \beta$
such that
\begin{equation}
 h_i|_\Gamma = \alpha h_e|_\Gamma \ \ \ \partial_n h_i|_\Gamma = \beta \partial_n h_e|_\Gamma.
\end{equation}
\end{proposition}

\begin{proof} For the proof we simply change $h_e$ into $\alpha h_e$, and remark that this is a
single layer potential of a charge $\rho$. The second proportionality condition implies,
via the jump formulas, $K^\ast S\rho = \lambda S\rho$ for a suitable $\lambda$. By the injectivity of
$S$ we find $K\rho = \lambda \rho$, and the general symmetrization framework implies
$\rho \in W^{1/2}(\Gamma)$.
\end{proof}

Note that in the above proof we could as well renormalize the
normal derivatives and assume that $(\beta h_i, h_e)$ is a double
layer potential of a density $f \in W^{1/2}(\Gamma)$ which turns
out to be an eigenfunction of the NP operator $K$.

\section{Essential spectrum}

A little more than a century ago Torsten Carleman has defended his doctoral dissertation titled "\"Uber das Neumann--Poincar\'esche Problem f\"ur ein Gebiet mit Ecken" \cite{Car}.  While the NP operator is compact on smooth boundaries, the presence of corners produces continua in its (essential) spectrum.
In a tour de force Carleman did solve the singular integral equation governed by the NP operator and analyzed the (asymptotic) structure of its solutions in a domain with one corner. He made use of elementary and very ingenious geometric transformations together with the, at his time new, theory of Fredholm determinants combined with the canonical factorization of entire functions due to Hadamard. Carleman's work, still puzzling in its fine detail, did not attract the visibility it deserves, nor
did the prior results of his predecessors, among which we mention Zaremba \cite{Zaremba}. Carleman returned all his life to spectral analysis questions related to integral equations with a singular kernel. The aim of the present section is to indicate some advances and proofs at the level of current mathematical rigor of Carleman's claims contained in his dissertation.

Only a few years after Carleman's defense, Radon \cite{Radon1,Radon2} developed the theory of measures of bounded variation, and applied it to study the NP operator on the space of continuous functions $C(\partial \Omega)$. He computed the essential spectral radius for boundaries $\partial \Omega$ possessing bounded rotation, validating Carleman's bound for the essential spectrum. Nowadays we know much more: for non-smooth boundaries, the spectrum of the NP operator depends drastically on the underlying space.  For instance, when the NP operator is considered on $L^p(\partial \Omega)$, $p \geq 2$, and $\Omega$ is a curvilinear polygon, the complete spectral picture is available \cite{Mitrea} - and it is entirely different from what appears in the works of Carleman and Radon. More specifically,  the essential spectrum of $K$ is no more contained in the real axis and consists of a collection of closed lemniscates whose shape depends on $p \geq 2$.

Two real dimensions are special, not last due to the presence of the complex variable. Returning to the energy space and the orthogonal decomposition
$$ {H} = {\mathcal S} \oplus {\mathcal D} = {H}_i \oplus {H}_e$$
with the associated orthogonal projections $P_i,P_e,P_s,P_d$, we note that the angle operators $(P_e-P_i)P_s$ and $(P_s-P_d)P_i$ are similar, hence they possess the same spectrum. This simple geometric observation leads to a completely different realization of the NP operator.

More precisely, for a bounded domain
$\Omega$ of the complex plane, define the anti-linear Ahlfors--Beurling operator $T_\Omega$ acting on the Bergman space $L^2_a(\Omega)$,
\begin{equation} \label{eq:beurlingtransform}
T_\Omega f (z) = \frac{1}{\pi} {\rm p. v.} \int_\Omega \frac{\overline{f(\zeta)}}{(\zeta-z)^2} \, dA(\zeta), \qquad f\in L^2_a(\Omega), \, z \in \Omega.
\end{equation}
Above Bergman space $L^2_a(\Omega)$ consists of the holomorphic, square-integrable functions in $\Omega$ with respect to the area measure $dA$.

The angle operators equality implies that $K^\ast \colon W_0^{-1/2} \left( \partial \Omega \right) \to W_0^{-1/2} \left( \partial \Omega \right)$ and $T_\Omega : L^2_a(\Omega) \to L^2_a(\Omega)$ are similar as $\mathbb{R}$-linear operators. The notation  $W_0^{-1/2}$ means all charges orthogonal to the constants.
From the similarity we derive the equality of spectra:
\[\sigma(K) = \sigma_\mathbb{R}(T_\Omega) \cup \{1\}.\]
 Note that if $\lambda$ is in the spectrum of $T_\Omega$, then by anti-linearity so is $e^{i\theta} \lambda$ for any $\theta \in \mathbb{R}$. However, we are interested only in the real spectrum, $K : W^{1/2}(\partial \Omega)/\mathbb{C} \to W^{1/2}(\partial \Omega)/\mathbb{C}$ being similar to a self-adjoint operator over the complex field. Therefore, to determine the spectrum of $K$ using $T_\Omega$ we consider only $\lambda \geq 0$ in the spectrum of $T_\Omega$, and note that the spectrum of $K$ consists of $\pm \lambda$, for all such points $\lambda$, in addition to the simple eigenvalue $1$.

 The existence of a harmonic conjugate of every harmonic function defined in a simply connected domains implies that the spectrum of the NP operator
 on a 2D domain is always symmetric with respect to the origin, with the exception of the simple eigenvalue $1$. This fact is even more transparent in the Ahlfors-Beurling interpretation, see \cite{KPS-ARMA-07}.

 By a happy coincidence, the skew eigenvalue problem
 $$ T_\Omega f = \lambda f, \ \ f \in L^2_a(\Omega)$$
 was well studied in function theory of a complex variables. The positive real eigenvalues are called {\it Fredholm eigenvalues} of the domain
 $\Omega$ and the eigenfunctions are orthogonal and complete in Bergman space (provided the boundary of $\Omega$ is smooth). The vast literature on the subject goes back to Ahlfors \cite{Ahl}, through the exhaustive works of Schiffer \cite{Sch1,Sch2,BS}, touching classical contributions in quasi-conformal mapping theory \cite{Springer}. For instance, Ahlfors \cite{Ahl} observed a connection between the spectral radius of the NP operator (the largest Fredholm eigenvalue of $\Omega$) and the quasiconformal reflection coefficient of $\partial \Omega$. The reflection coefficient is notoriously difficult to compute for general domains which do not have any special geometric structure. Yet, Ahlfors' inequality provides to date nearly all known spectral bounds of the NP operator in the energy norm.

 Combining conformal mapping methods with an exact computation on a wedge, and via the Ahlfors-Beurling realization of the NP operator, the essential spectrum on a bounded planar domain with finitely many corners in its boundary was proved  to be exactly the closed interval
 $$\sigma_{ess}(K, W^{1/2}) = [- 1+\frac{\alpha}{\pi}, 1 + \frac{\alpha}{\pi}],$$
 where $\alpha$ is the least inner angle in the boundary. For a proof see \cite{PP1,PP2}. This is exactly the bound calculated by Carleman via an explicit spectral resolution involving generalized eigenfunctions. Details of his insight remain to be clarified in today's terminology and more rigorous language. See \cite{KLY}.

 Very recent studies confirm and complement the above spectral picture, see \cite{BZ} for a natural Weyl sequence approach on a particular shape and the analysis carried in \cite{Perfekt2}.

 The estimation of the essential spectrum of the NP operator in higher dimensions is more dramatic from the computational point of view, and for this reason even more interesting \cite{HP1,HP2}.

\section{Decay estimates of NP-eigenvalues in two dimensions}\label{sec:decay}

{The present sections brings us very recent contributions referring to the spectral asymptotics of the NP operator in two dimensions.
Throughout this section we assume that the boundary $\Gamma$ is at least $C^{1, \alpha}$-smooth, with $\alpha>0$.
Therefore the NP operator is compact due to the fact that its integral kernel has only a weak singularity.
In this case the spectrum consists of  a sequence of eigenvalues $\{ \lambda_j \}_{j = 0}^\infty$ of finite multiplicities accumulating to $0$, which is the only element of the essential spectrum.}
We enumerate $\{ \lambda_j \}_{j = 0}^\infty$ in the following way:
\begin{equation} \label{enumerated_eigenvalues}
  \lambda_0 = 1  > | \lambda_1 | = | \lambda_2 | \ge \dots \ge | \lambda_{2 n - 1} | = | \lambda_{2 n} | \ge \to 0 \quad  \text{ as } j \to \infty.
\end{equation}

Very few explicit computations in closed form of the eigenvalues of the NP operator are known.
For example, if $\Gamma$ is a circle, we can easily see that $\lambda_j = 0$ ($j = 1, 2, \dots$), since the integral kernel is a constant function.
For ellipses, direct calculations show that the eigenfunctions decay exponentially (see \cite{Ahl,Korn1}).
More precisely, if $\Gamma$ is the ellipse with the long axis $a$ and short axis $b$, one finds from the conformal mapping of the exterior to the disk, that the NP-eigenvalues are:
\begin{equation} \label{ellipse_eigenvalues}
  \pm \left( \frac{a - b}{a + b} \right)^j, \quad j = 1, 2, \dots.
\end{equation}
In general, the decay rate of the eigenvalues of the NP operator depends on the smoothness of $\Gamma$. The following recent theorem due to Miyanishi and Suzuki \cite{MS} offers some precise information.
The ordered {\it singular values} of the NP operator, namely, eigenvalues of $(K^*K)^{1/2}$, are denoted:
\begin{equation*}
  \alpha_0 \ge \alpha_1 \ge \alpha_2 \ge \cdots \to 0.
\end{equation*}

\begin{theorem} \label{theorem_MS}
Let $k \geq 2$ and let $\Gamma$ be a $C^k$ Jordan curve in two real dimensions.
For any $\alpha > - 2 k + 3$,
\begin{equation*}
  \alpha_j = o(j^{\alpha / 2}) \text{ and } \lambda_j = o(j^{\alpha / 2}) \quad \text{ as } j \to \infty.
\end{equation*}
\end{theorem}
The proof relies on a result of Delgado and Ruzhansky \cite{DR} which links the regularity of the integral kernel of the Neumann-Poincar\'{e} operator to the Shatten-von Neumann class $\mathcal{C}^r(L^2(\Gamma))$, $r > 2/(2 k - 3)$.

If $\Gamma$ does not meet the smoothness assumption in Theorem~\ref{theorem_MS}, then the decay picture is more involved.
A recent work of Jung and Lim \cite{JungLim} elucidates this setting.
\begin{theorem} \label{theorem_JL}
Let $\Omega$ be a simply connected bounded domain having $C^{1+p, \alpha}$ boundary with $p \ge 0$, $\alpha \in (0, 1)$, and $p + \alpha > 1/2$.
Let $\{ \lambda_j \}_{j = 0}^\infty$ denote the eigenvalues of the associated NP operator, enumerated as \eqref{enumerated_eigenvalues}.
Then there exists a constant $C > 0$ independent of $j$ such that
\begin{equation*}
  \big| \lambda_{2j-1} \big| = \big| \lambda_{2j} \big| \le C j^{-p-\alpha+1/2} \quad \text{for all } j = 1, 2, 3, \dots.
\end{equation*}
\end{theorem}

The proof invokes the Riemann map $\Psi$ from the exterior of a disk to the exterior of $\Omega$.
More specifically, the conformal map $\Psi$ induces an orthogonal coordinate system outside $\Omega$ by push-forward of the polar coordinate system.
The eigenfunctions of the NP operator are expressed as trigonometric functions parametrized by the curvilinear coordinate system.
Rather explicit formulae for the eigenvalues are then derived, from which the stated decay rate follows.

Assuming $\Gamma$ is a $C^\infty$ Jordan curve we infer from the preceding theorem that $\lambda_j = o(j^{- \infty})$ as $j \to \infty$. However, in the case of an ellipse, we know from \eqref{ellipse_eigenvalues} that $\lambda_j$ decays exponentially fast as $j \to \infty$.
It is therefore expected that real analyticity of $\Gamma$ implies an exponential decay of the associated NP-eigenvalues.

In order to prove the later claim, we pass to complex coordinates and identify $\mathbb{R}^2$ with $\mathbb{C}$.
Let $S^1$ be the unit circle and $Q: S^1 \to \Gamma \subset \mathbb{C}$ be a regular parametrization of $\Gamma$.
Let
\begin{equation*}
q(t) := Q(e^{i t}), \quad t \in \mathbb{R}.
\end{equation*}
Then $q$ is real analytic if $\Gamma$ is real analytic.
Note that $q(t + 2 \pi) = q(t)$ for any $t \in \mathbb{R}$.

Suppose now that $\Gamma$ is real analytic.
Then $Q$ admits an extension as an analytic mapping from an annulus
\begin{equation*}
A_\epsilon := \{ \tau \in \mathbb{C} : e^{- \epsilon} < | \tau| < e^\epsilon\}
\end{equation*}
for some $\epsilon > 0$ onto a tubular neighborhood of $\Gamma$ in $\mathbb{C}$.
The function $q$ is an analytic function from $\mathbb{R} \times i (- \epsilon, \epsilon)$ onto a tublar neighborhood of $\Gamma$.

For a real analytic parametrization $q$ of $\Gamma$, we consider the numbers $\epsilon$ such that $q$ satisfies the following additional condition:
\begin{center}
(G) \ if $q(t) = q(s)$ for $t \in [- \pi, \pi] \times i (- \epsilon, \epsilon)$ and $s \in [- \pi, \pi)$, then $t = s$.
\end{center}
It is worth emphasizing that condition (G) is weaker than univalence.
It only requires that $q$ attains the value $q(s)$ for $s \in [- \pi, \pi)$ only at $s$.
The condition (G) is equivalent to the fact that the only points mapped by the function $q: \mathbb{R} \times i (- \epsilon, \epsilon) \to \mathbb{C}$ to $\Gamma$ are those on the real line.
Since $Q$ is one-to-one on $\Gamma$, the extended function is univalent in $A_\epsilon$ if $\epsilon$ is sufficiently small.
Therefore, the condition (G) is fullfied if $\epsilon$ is small.
We denote the supremum of such a threshold $\epsilon$ by $\epsilon_q$ and call it the {\it modified maximal Grauert radius} of $q$.
We emphasize that $\epsilon_q$ may differ depending on the parametrization $q$.
Let
\begin{equation*}
\epsilon_\Gamma := \sup_q \epsilon_q,
\end{equation*}
where the supremum is taken over all regular real analytic parametrization $q$ of $\Gamma$.
We call $\epsilon_\Gamma$ the {\it modified maximal Grauert radius} of $\Gamma$.

The following theorem \cite{AKM} holds.
\begin{theorem} \label{theorem_AKM}
Let $\Omega$ be a bounded planar domain with the analytic boundary $\Gamma$ and $\epsilon_\Gamma$ be the modified maximal Grauert radius of $\Gamma$. 
Let $\{ \lambda_j \}_{j = 0}^\infty$ be the eigenvalues of the NP operator enumerated as in \eqref{enumerated_eigenvalues}.
Then for any $\epsilon < \epsilon_\Gamma,$ there exists a constant $C$ with the property:
\begin{equation}
  | \lambda_{2 j - 1} | = | \lambda_{2 j} | \le C e^{- \epsilon j} \label{exponential_estimate}
\end{equation}
for all $j$.
\end{theorem}
The proof relies on the Weyl-Courant min-max principle (see \cite{LR}) and a Paley-Wiener type lemma.

For circles, the integral kernel is a constant function, which means that $\epsilon_\Gamma = \infty$.
From Theorem~\ref{theorem_AKM}, $| \lambda_j | \le C e^{- c j}$ for any $c > 0$.
Indeed, $\lambda_j = 0$ for $j \ge 1$. It is shown in \cite{AKM} that $\epsilon_q = \log{\frac{a + b}{a - b}}$ for the special regular parametrization $q$ of an ellipse. So, in view of \eqref{ellipse_eigenvalues}, the estimate \eqref{exponential_estimate} is optimal for ellipses. According to \cite{Sch1}, eigenvalues of the NP operator are invariant under the M\"{o}bius transformations.
Hence the eigenvalues of the NP operator for lima\c{c}ons of Pascal are identical to those for an ellipse.
The modified maximal Grauert radius of the lima\c{c}on is calculated in \cite{AKM}, which shows that \eqref{exponential_estimate} is optimal for the lima\c{c}on.

It is not known whether the estimate \eqref{exponential_estimate} is optimal or not in general, and if it is optimal, whether \eqref{exponential_estimate} holds for $\epsilon = \epsilon_\Gamma$ or not.

%%%%%%%%%%%%%%%%%%%%%%%%%%%%%%%%%%%%%%%%%%%%%%%%%%%%%%%%%%%%%%%%%%%%%%%%%%%%%%%%%%%%%%%%%%%%%%%
\section{Spectral properties of  the NP operator on three dimensional smooth domains}
%%%%%%%%%%%%%%%%%%%%%%%%%%%%%%%%%%%%%%%%%%%%%%%%%%%%%%%%%%%%%%%%%%%%%%%%%%%%%%%%%%%%%%%%%%%%%%%

As seen in the previous section, the decay rates of the NP eigenvalues in two dimensions differ depending on the smoothness of the boundary $\Gamma=\partial\Omega$. However, one dimension higher, the case $d = 3$ reveals a totally different picture, encoded in the so-called Weyl's law for NP-eigenvalues. The present section collects some very recent result along these lines.

Let $\Omega$ be a three dimensional $C^{1, \alpha}$ bounded domain.  As already mentioned, we know that $K$ is a compact operator on $L^2(\Gamma)$  (or $W^{1/2}(\Gamma)$) and the set of its eigenvalues consists of at most countable set of real numbers, with $0$ the only possible limit point. It is also known that the eigenvalues of the NP operator lie in the interval $(-1, 1]$ and  $1$ is the eigenvalue corresponding to the constant eigenfunction. Unlike the two-dimensional case, in three dimensions NP eigenvalues $\lambda_j$ behaves like $Cj^{-1/2}$ as $j \to \infty$ where $C$ is a constant determined by the surface geometry such as the Willmore energy and the Euler characteristic, as the following theorem shows. The {\it Willmore energy} on $\Gamma$ is defined to be
\begin{equation}
\label{definition of Willmore energy}
W=W(\Gamma):=\int_{\Gamma} H^2(x)\; dS_{x}
\end{equation}
where $H(x)$ is the mean curvature of the surface $\Gamma$.

The following result was obtained in \cite{Miyanishi:Weyl}:
\begin{theorem}\label{main}
Let $\Omega\subset{{\mathbb{R}}^3}$ be a $C^{2, \alpha}$ bounded domain for some $\alpha>0$.
Let $\lambda_j$ be NP eigenvalues on $\Gamma=\partial\Omega$ enumerated counting multiplicities in the following way
\begin{equation}
1=\lambda_0 > |\lambda_1| \geq |\lambda_2| \geq\; \cdots .
\end{equation}
Then
\begin{equation}\label{Weyl NP}
|\lambda_j| \sim 2\Big(\frac{3W - 2\pi \chi}{128 \pi} \Big)^{1/2}  j^{-1/2}\quad \text{as}\ j\rightarrow \infty,
\end{equation}
where $W$ and $\chi$ are the Willmore energy and the Euler characteristic of $\Gamma=\partial \Omega$, respectively.
\end{theorem}

We mention that
$$
\Big(\frac{3W(\Gamma) - 2\pi \chi(\Gamma)}{128 \pi} \Big)^{1/2} \geq \frac{1}{4},
$$
and that equality holds if and only if $\Gamma=S^2$. It is because $W(\Gamma)\geq 4\pi$ where the equality holds if and only if $\Gamma=S^2$ (See \cite{Ma-Ne} and reference therein) and
$\chi \leq 2$. Thus we infer from \eqref{Weyl NP} that the NP operator has infinitely many eigenvalues, and hence has infinite rank, which answers a question raised in \cite{KPS-ARMA-07}. Another interesting consequence is that the asymptotic behavior (\ref{Weyl NP}) is M\"obius invariant. It is so because the Willmore energy is invariant under M\"obius transformations of ${\mathbb{R}}^3$ \cite{Bl} and the Euler characteristic is topologically invariant. More consequences of (\ref{Weyl NP}) can be found in \cite{Miyanishi:Weyl}.

If $\Gamma=S^2$, namely, $\Omega$ is a ball, then Poincar\'e \cite{Poi1} proved that the NP eigenvalues are $1/(2k+1)$ for $k = 0, 1, 2\, \ldots$ and their multiplicities are $2k +1$. We may enumerate them as
$$
{\underbrace{\frac{1}{1}}_{1}, \underbrace{\frac{1}{3}, \frac{1}{3}, \frac{1}{3}}_{3}, \underbrace{\frac{1}{5}, \frac{1}{5}, \cdots, \frac{1}{5}}_{5}, \cdots, \underbrace{\frac{1}{(2k+1)}, \cdots, \frac{1}{(2k+1)}}_{2k+1}}, \cdots,
$$
and see that the $j=k^2$th eigenvalue satisfies
$$
\lambda_j=\frac{1}{2k+1} \sim \frac{1}{2} j^{-1/2}.
$$
Since ${W(S^2) =4\pi}$ and $\chi(S^2)=2$, it coincides with \eqref{Weyl NP}

The NP operator on a smooth surface is a pseudo-differential operator of order $-1$ and the known formulas for the eigenvalue asymptotics of pseudo-differential operators in \cite{BiS} can be applied with a few modifications. For less smooth case, namely, the case of a $C^{2,\alpha}$ smooth boundary, an additional perturbation argument is applied  to derive \eqref{Weyl NP}.

The spectral structure of the NP operator on a non-convex domain in $\mathbb{R}^3$ seems quite different from that on convex domains. As we see above, the NP eigenvalues on balls are all positive. A negative NP eigenvalue was found an oblate spheroid in \cite{Ah2}. However, existence of negative NP eigenvalues has been a long standing question. Recently, it is proved that if there is a point of non-convexity on the boundary of the domain, then the NP operator on either the domain or its inversion should have a negative eigenvalue \cite{JiKang}. It is also proved recently in \cite{AKJiKKM} that the NP operator on the standard torus has infinitely many negative eigenvalues as well as infinitely many positive ones.

We now present recent results on Weyl-type asymptotics of NP eigenvalues on non-convex domains. For that, let $\lambda_j^+$ be positive NP eigenvalues enumerated in descending order counting multiplicities, namely,
\begin{equation}\label{spec}
1=\lambda_0^{+} >  \lambda_1^{+} \geq \lambda_2^{+} \geq \dots ,
\end{equation}
and let $\lambda_j^-$ be negative NP eigenvalues enumerated in ascending order,
\begin{equation}\label{spec2}
\lambda_1^{-} \le \lambda_2^{-} \le \dots.
\end{equation}
It is \emph{a priori} possible that the set of negative eigenvalues is finite or even void.

The following theorem is obtained in \cite{Miyanishi-Rozenblum}.

\begin{theorem}\label{mainPM}
Let $\Omega\subset {\mathbb{R}}^3$ be a bounded domain with $C^\infty$ boundary. Then
\begin{equation}\label{AsPM}
    \lambda^{\pm}_j\sim\ \pm A_{\pm}(\Gamma)^{\frac12}j^{-\frac12},\ \  j\to\infty,
\end{equation}
where
\begin{equation}\label{APM}
    A_{\pm}(\Gamma)=\frac{1}{32\pi^2}\int_{\Gamma}{dS_{x}}\int_0^{2\pi}[(k_1({x}) \cos^2\theta +k_2({x})\sin^2\theta)_{\mp}]^2d\theta.
\end{equation}
Here $k_1({x}), k_2({x})$ are the principal curvatures of the surface $\partial\Omega$ at the point ${x}$ under the direction of the normal vector chosen to be the exterior one. The subscripts $\mp$ in \eqref{APM}
respectively indicate the positive and negative parts, namely,
$$
z_+=\max (z,0),\quad
z_-=\min (z, 0).
$$
If the coefficient $A_+(\Gamma)$ or $A_-(\Gamma)$ turns out to be zero, formula \eqref{AsPM} should be understood as $\lambda^{\pm}_j=o(j^{-\frac12})$.
\end{theorem}

Note that
$$
A_+(\Gamma)+A_-(\Gamma)= \frac{3W(\Gamma) - 2\pi \chi(\Gamma)}{32 \pi}.
$$
Theorem \ref{mainPM} shows that if $\Gamma=\partial \Omega$ is infinitely smooth, then there \emph{always} exist infinitely many positive eigenvalues of the NP operator. If there exists a point where at least one of the principal curvatures is positive so that $\Omega$ is not convex near that point, there exist infinitely many negative eigenvalues.

The following theorem is also obtained in \cite{Miyanishi-Rozenblum}. There, the uniform convexity means that the principal curvatures are everywhere negative, which implies that they are separated from zero due to the smoothness assumption.

\begin{theorem}\label{finitenegative}
Let $\Omega\subset{\mathbb{R}}^3$ be a bounded domain with smooth boundary $\partial\Omega$ which is uniformly convex at all points. Then the associated NP operator possesses at most finitely many negative eigenvalues.
\end{theorem}

We emphasize that
Theorem \ref{mainPM} and Theorem \ref{finitenegative} have been proved only for $C^{\infty}$ surfaces, while Theorem \ref{main} is proved for $C^{2, \alpha}$ surfaces. We conjecture that the results presented in this section hold true for $C^{1, 1}$ surfaces \cite{Miyanishi-Rozenblum}.

\section{Elastic NP operator}

Let $\Omega$ be a bounded domain in $\mathbb{R}^d$ whose boundary $\Gamma=\partial\Omega$ is Lipschitz continuous. Let $(\lambda, \mu)$ be the Lam\'e constants for $\Omega$ satisfying the strong convexity condition: $\mu > 0$ and $d\lambda + 2\mu > 0$. The isotropic elasticity tensor $\mathbb{C} = ( C_{ijkl} )_{i, j, k, l = 1}^3$ and the corresponding Lam\'e system $\mathcal{L}_{\lambda,\mu}$ are defined by
\begin{equation}
C_{ijkl} := \lambda \, \delta_{ij} \delta_{kl} + \mu \, ( \delta_{ik} \delta_{jl} + \delta_{il} \delta_{jk} )
\end{equation}
and
\begin{equation}
\mathcal{L}_{\lambda,\mu} u := \nabla \cdot \mathbb{C} \widehat\nabla u = \mu \Delta u + (\lambda + \mu) \nabla \nabla \cdot u,
\end{equation}
where $\widehat\nabla$ denotes the symmetric gradient, namely,
\begin{equation}
\widehat\nabla u := \frac{1}{2} \left( \nabla u + \nabla u^T \right) \quad (T \mbox{ for transpose}). \nonumber
\end{equation}
The corresponding conormal derivative on $\partial \Omega$ is defined to be
\begin{equation}\label{conormal}
\partial_\nu u := (\mathbb{C} \widehat\nabla u) n = \lambda (\nabla \cdot u) n + 2\mu (\widehat\nabla u) n \quad \mbox{on } \partial \Omega,
\end{equation}
where $n$ is the outward unit normal to $\partial \Omega$.

Let $\mathbf{G} = \left( G_{ij} \right)_{i, j = 1}^d$ is the Kelvin matrix of fundamental solutions to the Lam\'{e} operator $\mathcal{L}_{\lambda, \mu}$, namely,
\begin{equation}\label{Kelvin}
  G_{ij}(\mathbf{x}) =
  \begin{cases}
    - \displaystyle \frac{\alpha_1}{4 \pi} \frac{\delta_{ij}}{|\mathbf{x}|} - \frac{\alpha_2}{4 \pi} \displaystyle \frac{x_i x_j}{|\mathbf{x}|^3}, & \text{ if } d = 3, \\
    \noalign{\smallskip}
    \displaystyle \frac{\alpha_1}{2 \pi} \delta_{ij} \ln{|\mathbf{x}|} - \frac{\alpha_2}{2 \pi} \displaystyle \frac{x_i x_j}{|\mathbf{x}| ^2}, & \text{ if } d = 2,
  \end{cases}
\end{equation}
where
\begin{equation}
  \alpha_1 = \frac{1}{2} \left( \frac{1}{\mu} + \frac{1}{2 \mu + \lambda} \right) \quad\mbox{and}\quad \alpha_2 = \frac{1}{2} \left( \frac{1}{\mu} - \frac{1}{2 \mu + \lambda} \right).
\end{equation}
The NP operator for the Lam\'e system is defined by
\begin{equation}\label{BK}
\mathbf{K} [\mathbf{f}] (\mathbf{x}) := \mbox{p.v.} 2\int_{\partial \Omega} \partial_{\nu_\mathbf{y}} {\bf G} (\mathbf{x}-\mathbf{y}) \mathbf{f}(\mathbf{y}) d \sigma(\mathbf{y}) \quad \mbox{a.e. } \mathbf{x} \in \partial \Omega.
\end{equation}
Here, p.v. stands for the Cauchy principal value, and the conormal derivative $\partial_{\nu_\mathbf{y}}\mathbf{G} (\mathbf{x}-\mathbf{y})$ of the Kelvin matrix with respect to $\mathbf{y}$-variables is defined by
\begin{equation}\label{kerdef}
\partial_{\nu_\mathbf{y}}\mathbf{G} (\mathbf{x}-\mathbf{y}) \mathbf{b} = \partial_{\nu_\mathbf{y}} (\mathbf{G} (\mathbf{x}-\mathbf{y}) \mathbf{b})
\end{equation}
for any constant vector $\mathbf{b}$. Again the operator $\mathbf{K}$ is 2 times that in the papers cited in this section.

Unlike the electro-static case, the elastic NP operator is \emph{not} compact even on smooth domains, which is due to the off-diagonal entries of the Kelvin matrix $\mathbf{G}(x)$ (see \cite{DKV-Duke-88}), and the spectral properties of the elastic NP operator were not known. It turns out that the elastic NP operator is polynomially compact if $\Gamma$ is smooth, and has continuous spectrum if $\Gamma$ has a corner, as we describe in the sequel.

With the pair of Lam\'e parameters $\mu$ and $\lambda$, let
\begin{equation}\label{kzero}
k_0= \frac{\mu}{2\mu+\lambda},
\end{equation}
and define two polynomials $p_2$ and $p_3$ by
\begin{equation}
p_2(t)=t^2 - k_0^2, \quad p_3(t)=t(t^2-k_0^2).
\end{equation}
Then the following theorem holds:

\begin{theorem}\label{Polynomially compact}
Let $\Omega$ be a bounded domain in $\mathbb{R}^d$ with $C^{1,\alpha}$-smooth boundary for some $\alpha>0$.
\begin{itemize}
\item[{\rm (i)}] If $d=2$, then $p_2(\mathbf{K})$ is compact. Moreover, $\mathbf{K}-k_0$ and $\mathbf{K}+k_0$ are \emph{not} compact.
\item[{\rm (ii)}] If $d=3$, then $p_3(\mathbf{K})$ is compact. Moreover, $\mathbf{K} (\mathbf{K}-k_0 \mathbf{I})$, $\mathbf{K} (\mathbf{K} + k_0 \mathbf{I})$ and $\mathbf{K}^2 - k_0^2 \mathbf{I}$ are \emph{not} compact.
\end{itemize}
\end{theorem}

The two-dimensional result was proved in \cite{AJKKY18} using the Hilbert transform, and the three-dimensional result was proved for $C^\infty$-smooth domains in \cite{AKM19} using the pseudo-differential calculus, and extended to $C^{1,\alpha}$ domains in \cite{KK-arXiv} by computing the composition of the surface Riesz transforms.

As a consequence of Theorem \ref{Polynomially compact} and the spectral mapping theorem, the following theorem on the spectral structure of the elastic NP operator was obtained:

\begin{theorem}
Let $\Omega$ be a bounded domain in $\mathbb{R}^d$ with $C^{1,\alpha}$-smooth boundary for some $\alpha>0$.
\begin{itemize}
\item[{\rm (i)}] If $d=2$, $\sigma(\mathbf{K})$ consists of two non-empty sequences of eigenvalues  which converge to $k_0$ and $-k_0$, respectively.
\item[{\rm (ii)}] If $d=3$, $\sigma(\mathbf{K})$ consists of three non-empty sequences of eigenvalues  which converge to $0$, $k_0$ and $-k_0$, respectively.
\end{itemize}
\end{theorem}

If $\Omega$ is a disk, then $\sigma(\mathbf{K})$ consists of $k_0$ and $-k_0$ \cite{AJKKY18}. The spectrum of the ball has been computed in \cite{DLL19}.

A natural follow-up question is about the convergence rate of eigenvalues to $0$. Regarding this question, the following theorems for smooth and real analytic boundaries in {\it two dimensions} have been proved in \cite{AKM-arXiv}.

\begin{theorem}
If $\Gamma$ is $C^{k, \alpha}$ with $k+\alpha > 2$ and $0\leq \alpha <1$, then eigenvalues $\lambda_j^{\pm}$ of the elastic NP operator $\mathbf{K}$ converging to $\pm k_0$ satisfy
\begin{equation}
\lambda_j^{\pm} = \pm k_0  + o(j^{d}) \quad \text{as}\; j\rightarrow \infty
\end{equation}
for any $d>-(k+\alpha)+3/2$.
\end{theorem}

To describe the result for domains with real analytic boundaries, we need the following notion in addition to the notion of the modified maximal Grauert radius defined before. Let $\Phi:U \to \Omega$ be a Riemann mapping and let $q$ be the parametrization of $\Gamma$ by $\Phi$, namely, $q(s)=\Phi(e^{is})$. Such a parametrization is called a parametrization of $\Gamma$ by a Riemann mapping.

\begin{theorem}
Suppose that $\Gamma$ is real analytic. Let $q$ be a parametrization of $\Gamma$ by a Riemann mapping and let $\epsilon_q$ be its modified maximal Grauert radius. Then, eigenvalues $\lambda_j^{\pm}$ of the elastic NP operator $\mathbf{K}$ converging to $\pm k_0$ satisfy
\begin{equation}
\lambda^{\pm}_j = \pm k_0+ o(e^{-\epsilon j}) \quad \text{as}\; j\rightarrow \infty
\end{equation}
for any $\epsilon <\epsilon_q/8$.
\end{theorem}

When two-dimensional domain $\Omega$ has a corner on its boundary, then it is proved in the forthcoming paper \cite{BDKK} that there are intervals of continuous spectra around $-k_0$ and $k_0$.

\section{Spectral analysis in a space with two norms}\label{2-norms}

On the abstract spectral analysis side, it was Mark G. Krein who found in 1947 (and probably much earlier) a theoretical explanation for the spectral behavior of ``symmetrisable compact linear transformations", a term already used
in potential theory by a good dozen of authors. For an authoritative account of these early efforts see the chapter with the same title in Zaanen's book \cite{Za}. Krein's arguments are simple, irreducible in
their beauty and applicability, and have been rediscovered by P. Lax, J. Dieudonn\'{e}, and possibly others, see for references \cite{KPS-ARMA-07,Lax,Reid}.

In a couple of recent articles \cite{KP,Pu} we have adapted Krein's technique to a wider class of non-compact linear operators. In particular we have proved the spectral permanence in the case of the Lam\'e system of equations, between Lebesgue space and the weaker energy space of Sobolev fractional order.

Second, we have outlined in the same articles \cite{KP,Pu} a general scheme of adapting the general Galerkin approximation to the abstract setting of Neumann-Poincar\'e operator spectral analysis.  Specifically, we studied the simultaneous finite rank approximation of an operator acting on a Hilbert space endowed with a weaker norm.  A detailed analysis of the norm gap between the two sequences of finite central truncations of a linear bounded operator with respect to an ascending nest of finite dimensional subspaces isolates some specific choices of finite rank approximations with the same asymptotics with respect to the two norms. For general references on finite central truncations, known also under the name of the moment method, see \cite{B,LS}.

We reproduce below, for the enjoyment of the reader, Krein's main spectral permanence observations.
 Let $H$ be a complex separable Hilbert space and $T \in {\mathcal L}(H)$
a linear bounded operator acting on $H$. The spectrum of $T$ is denoted $\sigma(T)$.

We endow $H$ with a weaker pre-hilbertian space norm:
$$ (x,y) = \langle Ax, y \rangle, \ \ \ \ x, y \in H,$$
where $A >0$ is a positive, non-invertible bounded linear operator acting on $H$. If $A$ is invertible, then the two norms are equivalent, a much simpler scenario.
Let $K$ denote
the Hilbert space completion of $H$ with respect to the new norm. We have $H \subset K$, with dense range inclusion.

It is known since Krein's landmark article \cite{Krein} that an operator $T \in {\mathcal L}(H)$ which is symmetric with respect to the weaker norm
is automatically bounded with respect to it. Moreover, he proved in the same work that the additional compactness assumption on $T \in {\mathcal L}(H)$
implies the compactness and spectral permanence of $T$, with respect to the weaker norm.

\begin{proposition} \label{continuity} Assume that two linear bounded operators $S,T \in {\mathcal L}(H)$ satisfy
$$ AT = SA.$$
Then $T$ extends to a bounded linear transform of the Hilbert space $K$, that is
there exists a linear bounded map $M \in {\mathcal L}(H)$ satisfying
$$ \sqrt{A} T = M \sqrt{A}.$$
\end{proposition}

\begin{proof} Since $T^\ast A = A S^\ast$, the operators $X = T+S^\ast, Y = i(T-S^\ast)$ satisfy
$$ AX = X^\ast A, \ AY=Y^\ast A.$$ Then, according to Krein's observation \cite{Krein},  both $X$ and $Y$ extend continuously to the Hilbert space $K$.
We repeat the very simple proof. For every vector $x \in H$:
$$ \langle AXx, Xx\rangle^2  = \langle Ax, X^\ast x \rangle^2 \leq \langle Ax,x\rangle \langle A X^\ast x, X^\ast x\rangle.$$
Consequently, leaving aside the trivial case of vanishing denominators:
$$ \frac{(Xx,Xx)}{(x,x)} \leq \frac{(X^2 x, X^2x)}{(Xx,Xx)} \leq \ldots \leq \frac{(X^n x, X^n x)}{(X^{n-1}x, X^{n-1}x)}.$$
The product of all factors is telescopic, and yields
$$ [ \frac{(Xx,Xx)}{(x,x)} ]^n \leq \frac{(X^n x, X^n x)}{(x,x)} \leq \frac{\|A\| \| X \|^{2n}}{(x,x)},$$
hence
$$  \frac{(Xx,Xx)}{(x,x)} \leq \liminf_n [\frac{\|A\| \| X \|^{2n}}{(x,x)}]^{1/n} = \|X\|^2.$$
Thus $X$ extends to a linear bounded operator of the space $K$, with norm not exceeding the norm on $H$.
Then $T = \frac{X-iY}{2}$ turns to be bounded on $K$, that is
$$ \langle A Tx, Tx \rangle \leq C \langle Ax, x \rangle, \ \  x \in H,$$
for a universal positive constant $C$. Or equivalently,
$$ T^\ast A T \leq C A$$
in the operator norm. That is there exists a linear bounded map $M_1 \in {\mathcal L}(H)$ satisfying
$$ \sqrt{A}T = M_1 \sqrt{C} \sqrt{A},$$
and we can absorb the constant into the intertwiner: $M = M_1 \sqrt{C}.$
\end{proof}

Note that $T$ is symmetric with respect to the inner space $K$ if and only if $M = M^\ast$ as an operator of $H$.

\begin{corollary} The following spectral inclusions hold:
$$ \sigma_p(T, H) \subset \sigma_p(M, H),$$
$$ \sigma_{ap}(T, K) = \sigma_{ap} (M, H).$$
\end{corollary}

\begin{proof} The first inclusion is obvious. For proving the second one, let $x_n \in H$ be a sequence
of vectors, normalized in the weak norm: $(x_n, x_n ) = \langle A x_n, x_n \rangle = 1, \ \  n \geq 1.$
If $\lim_n (Tx_n, Tx_n) = 0$, then and only then \\
$\lim_n \| M \sqrt{A} x_n \| =0$. This proves that the
point $\lambda=0$ belongs to the approximative point spectrum of $T \in {\mathcal L}(K)$ if and only if it
belongs to the approximative point spectrum of $M \in {\mathcal L}(H)$.
\end{proof}

\begin{theorem} Let $T \in {\mathcal L}(H)$ be a symmetric operator with respect to the second inner product
$AT = T^\ast A$ and assume that the point spectrum is a discrete subset in the complement
of the essential spectrum of $T \in {\mathcal L}(H)$.Then
the point spectrum of the extension $T \in {\mathcal L}(K)$ is real, and equal to the point spectrum of $T \in {\mathcal L}(H)$,
with equal multiplicities, respectively.
\end{theorem}

\begin{proof} By a {\it root vector}, corresponding to the eigenvalue $\lambda$,  we mean a non-zero element of $\ker (T-\lambda)$.
Due to the symmetry of $T$ with respect to the inner product $(\cdot, \cdot)$, we have $\ker (T-\lambda) \neq 0$ only for real
values of $\lambda$. We will prove that the root subspaces of $T \in {\mathcal L}(H)$ and $T \in {\mathcal L}(K)$ coincide.

By assumption, every element $\lambda \in \sigma_p(T, H)$ does not belong to the essential spectrum of $T$, that is
$\ker (T-\lambda) $ is finite dimensional and ${\rm ran}(T-\lambda)$ is a closed subspaces of $H$, of finite codimension.
Let  $$V_\lambda = \{ x \in H, \ (x, y) = 0, y \in \ker (T-\lambda, H)\}.$$
Notice that $V_\lambda$ is a closed, finite codimensional subspace of $H$, invariant under $T$. Moreover,
$\ker (T-\lambda, V_\lambda) = 0$ and $(T-\lambda)V_\lambda = V_\lambda$ by the invariance of the Fredholm index under
finite rank perturbations. Then the operator $(T-\lambda, V_\lambda)^{-1}$ is bounded on $V_\lambda$ and
it is also symmetric, hence bounded by the above proposition in the weak norm of the space
$K \ominus \ker (T-\lambda, H)$. In conclusion, $\ker (T-\lambda, H) = \ker (T-\lambda, K)$.

\end{proof}

\begin{corollary} Assume, in the conditions of the theorem, that $\sigma_p(T, H)$ is dense in $\sigma(T, H)$.
Then $\sigma(T, H) = \sigma(T, K)$.
\end{corollary}

\begin{proof} We proved the equality of spectra  $\sigma_p(T, H) = \sigma_p(T,K)$ and we also know
from Proposition \ref{continuity} that $\sigma(T, K) \subset \sigma(T,H)$.
\end{proof}

\section{Concluding remarks} We briefly comment some recent advances related to the analysis of the Neumann-Poincar\'e operator, not included in our survey.

First of all, we mention several recent monographs devoted to the subject \cite{AK-book-07,AKL,Mazya,MM,AFKRYZ}.

One of the classical techniques of spectral analysis is to study the growth rate of the resolvent, possibly localized at a vector, in the neighborhood of en element of the point spectrum.
Even an absolute continuous part of the spectrum can be detected from the behavior of the resolvent. In certain applications this method bears the name of {\it resonance}. A great deal of effort was recently put into studying the resonance of specific boundary value problems, or interface problems, having the resonance of the NP operator as central theme \cite{AK14,ACKLM}.

The investigation of the NP operator on boundaries with minimal regularity, such as Lipschitz hypersurfaces, has occupied several generations of hard analysts, with definitive results
\cite{DKV-Duke-88}. See also \cite{Mazya,MM}.

Estimating the essential spectrum of the NP operator on domains with singularities in their boundary, such as corners, edges, cones, requires a formidable machinery of adapted pseudo-differential calculus and Fredholm theory of groupoid C*-algebras. This is a very active area of current research, under the leadership of Victor Nistor \cite{CNQ,MN}. On this ground, refined numerical experiments complement and enhance the theoretical findings \cite{HP1,HP2}.

An array of inverse problems rely on data carried on or transformed by the NP operator \cite{Kang,APST}. The celebrated Calderon problem, of detecting the conductivity of a medium
from field measurements on the boundary is closely connected to the theory of the NP operator. Compare also with Steklov's problem \cite{GP}.

A rather unexpected interpretation of the double layer potential in two real variables provided better estimates for the operator norm $\| p(T) \|$
of a polynomial applied to a complex matrix $T$, in terms of the uniform norm of $p$ on the numerical range $W(T)$ of $T$. In operator theory and numerical matrix analysis this
is an important variation on the von Neumann inequality theme.
A conjecture attributed to Crouzeix asserts that
the best constant $C$ in such an estimate
$$ \| p(T) \| \leq C \| p \|_{\infty, W(T)}, \ \ p \in {\mathbb C}[z],$$
should be $C=2$. So far it is known that $C = 1 + \sqrt{2},$ \cite{CP}.  For the role of the NP operator in this quest see also \cite{PS}.

The computational and numerical analysis aspects of layer potentials have evolved in parallel to the theoretical advances of the last decades. As customary nowadays, some spectral analysis phenomena referring to the NP operator were first discovered via numerical experiments, see in particular \cite{HP1,HP2}. From the abundant bibliography covering singular integral equations closely related to layer potentials we mention only a few titles: \cite{SW,ABN,BK,CKK,HO}.

\end{document}